\documentclass[11pt]{amsart}

\usepackage{amsmath,amssymb,amsfonts,amsthm,mathtools}
\usepackage{enumitem}
\usepackage[margin=1.15in]{geometry}
\usepackage[hidelinks]{hyperref}
\usepackage{microtype}

\allowdisplaybreaks

\newtheorem{theorem}{Theorem}[section]
\newtheorem{proposition}[theorem]{Proposition}
\newtheorem{lemma}[theorem]{Lemma}
\newtheorem{corollary}[theorem]{Corollary}
\newtheorem{conjecture}[theorem]{Conjecture}
\theoremstyle{definition}

\theoremstyle{remark}
\newtheorem{remark}[theorem]{Remark}
\newtheorem{example}[theorem]{Example}

\newcommand{\C}{\mathbb C}
\newcommand{\R}{\mathbb R}
\newcommand{\CP}{\mathbb{CP}}
\newcommand{\Ric}{\operatorname{Ric}}
\newcommand{\tr}{\operatorname{tr}}
\newcommand{\Vol}{\operatorname{Vol}}
\newcommand{\eps}{\varepsilon}
\newcommand{\cC}{\mathcal C}

\newcommand{\wt}{\widetilde}

\newcommand{\ov}{\overline}
\newcommand{\ddbar}{\partial\bar\partial}

\title[Constant $k$th-mixed curvature on LCK manifolds]
{Constant $k$th-mixed curvature on locally conformal K\"ahler manifolds}

\author{KAI TANG}
\address{School of Mathematical Sciences, Zhejiang Normal University, Jinhua, Zhejiang 321004, China}
\email{kaitang001@zjnu.edu.cn}

\author{ZUOCAI WANG}
\address{School of Mathematical Sciences, Zhejiang Normal University, Jinhua, Zhejiang 321004, China}
\email{202520500501@zjnu.edu.cn}

\thanks{\text{Foundation item:} Supported by Natural Science Foundation of Zhejiang Province (No. LMS26A010006).}
\begin{document}

\begin{abstract}
In this paper, we consider compact locally conformal K\"ahler manifolds with constant $k$th-mixed curvature. By using a recent method of Huang-Wan \cite{HuangWan} for constant Chern holomorphic sectional curvature, we prove that if a compact LCK manifold has nonzero constant $k$th-mixed curvature, then its Hermitian metric is K\"ahler. For the second mixed curvature, the same conclusion also holds when the curvature constant is zero. We also study some special parameters related to the general constant mixed curvature conjecture. In particular, we obtain a torsion-energy identity and a sign obstruction for compact Hermitian manifolds, and characterize an exceptional K\"ahler case by Bochner--K\"ahler geometry.
\end{abstract}

\maketitle

\section{Introduction}

The holomorphic sectional curvature is an important curvature quantity in complex differential geometry. It is well known that a complete K\"ahler manifold with constant holomorphic sectional curvature is a complex space form. Its universal cover is, up to a scaling, the complex projective space, the complex Euclidean space, or the complex hyperbolic space. For a general Hermitian metric, the Chern curvature tensor does not have all the K\"ahler symmetries, and the holomorphic sectional curvature does not determine the whole curvature tensor. This leads to the following well-known conjecture in non-K\"ahler geometry.

\begin{conjecture}[Constant holomorphic sectional curvature conjecture]\label{conj:HSC}
Let $(M^n,g)$ be a compact Hermitian manifold, $n\geq2$, whose Chern holomorphic sectional curvature is a constant $c$. If $c\neq0$, then $g$ is K\"ahler; if $c=0$, then $g$ is flat.
\end{conjecture}

The compactness assumption in Conjecture \ref{conj:HSC} is essential \cite{CCN}. In complex dimension two, the constant holomorphic sectional curvature conjecture has been completely solved by Balas--Gauduchon and Apostolov--Davidov--Mu\v{s}karov. More precisely, Balas--Gauduchon \cite{BG} proved the conjecture for nonpositive constants, while Apostolov--Davidov--Mu\v{s}karov \cite{ADM} treated the remaining cases by using their classification of compact self-dual Hermitian surfaces. In higher dimensions the conjecture is still open in general. It has been verified under several additional geometric assumptions, including Chern K\"ahler-like metrics, complex nilmanifolds, locally conformal K\"ahler metrics with nonpositive curvature constant, Strominger K\"ahler-like metrics, and some Bismut torsion-parallel metrics \cite{TangKL,LiZheng,CCN,RaoZheng,ChenZhengBTP}. Broder--Tang \cite{BroderTang} obtained several rigidity results for vanishing Hermitian curvature; in particular, a pluriclosed Hermitian metric with vanishing holomorphic sectional curvature on a compact K\"ahler manifold is K\"ahler. Recently, Huang--Wan \cite{HuangWan} proved Conjecture \ref{conj:HSC} for compact locally conformal K\"ahler manifolds without a sign assumption on $c$.

The constant holomorphic sectional curvature problem also gives a natural motivation for studying mixed curvature. In the K\"ahler setting, Chu--Lee--Tam \cite{CLT} introduced the following mixed curvature:
\[
\cC_{\alpha,\beta}(X)
=\frac{\alpha}{|X|^2}\Ric(X,\ov X)+\beta H(X),
\]
where $\alpha,\beta\in\R$. With this convention, $\alpha=0$ gives a multiple of the holomorphic sectional curvature, while $\beta=0$ gives a multiple of the first Chern Ricci curvature. Thus Conjecture \ref{conj:HSC} is the holomorphic sectional curvature case of the constant mixed curvature problem. On the other hand, a nonzero constant first Chern Ricci curvature forces a Hermitian metric to be K\"ahler, since the first Chern Ricci form is closed. It is therefore natural to consider constant linear combinations of these two curvature quantities.

For suitable choices of the parameters, mixed curvature also contains several familiar K\"ahler curvature conditions. It has been used to study positivity, projectivity and rational connectedness \cite{CLT,TangBLMS}. The mixed curvature was later studied for general Hermitian metrics, and the following extension of Conjecture \ref{conj:HSC} was proposed in \cite{TangPJM}.

\begin{conjecture}\label{conj:mixed}
Let $(M^n,g)$ be a compact Hermitian manifold, $n\geq2$, satisfying
\[
\cC_{\alpha,\beta}\equiv c,
\qquad \beta\neq0.
\]
If $c\neq0$, then $g$ is K\"ahler.
\end{conjecture}

For a general Hermitian metric, there are four natural traces of the Chern curvature tensor. We denote the corresponding Chern Ricci tensors by $\Ric^{(k)}$, $1\leq k\leq4$. Chen--Tang \cite{ChenTang} introduced the $k$th-mixed curvature:
\begin{equation}\label{eq:def-kmixed-intro}
\cC^{(k)}_{\alpha,\beta}(X)
=\frac{\alpha}{|X|^2}\Ric^{(k)}(X,\ov X)+\beta H(X),
\qquad \beta\neq0,
\end{equation}
and formulated the corresponding $k$th-mixed curvature conjecture.

\begin{conjecture}\label{conj:kmixed}
Let $(M^n,g)$ be a compact Hermitian manifold, $n\geq2$, with
\[
\cC^{(k)}_{\alpha,\beta}\equiv c,
\qquad \beta\neq0.
\]
If $c\neq0$, then $g$ is K\"ahler.
\end{conjecture}

For $k=1$, Tang \cite{TangPJM} proved Conjecture \ref{conj:kmixed} in complex dimension two. Chen--Zheng \cite{ChenZhengMixed} verified Conjecture \ref{conj:mixed} for several classes of Lie--Hermitian manifolds and for a class of balanced threefolds with parallel Bismut torsion. Chen--Tang \cite{ChenTang} proved that every compact Hermitian surface with constant $k$th-mixed curvature is self-dual, and that constant second mixed curvature forces a Hermitian surface to be K\"ahler. The higher-dimensional problem is still open in general.

In this paper, we consider the locally conformal K\"ahler case. Recall that a Hermitian metric $h$ is locally conformal K\"ahler, or LCK, if
\[
d\omega_h=\theta\wedge\omega_h,\qquad d\theta=0,
\]
where $\theta$ is the Lee form. Tang \cite{TangPJM} obtained partial results for constant mixed curvature on compact LCK manifolds under sign conditions on the parameters and the curvature constant. A recent theorem of Huang--Wan \cite{HuangWan} gives another approach to this problem.

We briefly recall the main idea of their method. On the universal covering $\pi:\wt M\to M$ of an LCK manifold one can write
\[
\pi^*h=e^{2F}g,
\]
where $g$ is K\"ahler. For constant Chern holomorphic sectional curvature, Huang--Wan \cite{HuangWan} derive an algebraic curvature identity for $g$ in which the additional terms are generated by the complex Hessian of $F$ through the standard $L$-operator. Therefore these terms have no Bochner component, and the K\"ahler metric $g$ is Bochner--K\"ahler. The globally conformal K\"ahler case is then treated using the rigidity of compact Bochner--K\"ahler metrics, while the strict LCK case is excluded by Kamishima's uniformization theorem together with the automorphy of the conformal factor. In this way, the sign conditions in the earlier integral approach can be avoided. Chen--Zheng and Chen--Li adapted the same idea to Levi-Civita and Bismut holomorphic sectional curvature, and to more general canonical metric connections \cite{ChenZhengLCK,ChenLi}.

The first purpose of this paper is to show that the method of Huang--Wan also works for all four $k$th-mixed curvatures. The main observation is that after passing to the universal K\"ahler cover, all the extra terms contributed by the four Chern Ricci tensors still lie in the Ricci and scalar summands of the K\"ahler curvature decomposition. More precisely, set
\begin{equation}\label{eq:eps-intro}
\eps_1=n\alpha+\beta,\qquad
\eps_2=\beta,\qquad
\eps_3=\eps_4=\alpha+\beta.
\end{equation}
If $A=\ddbar F$, the constant $k$th-mixed curvature equation can be written as
\begin{equation}\label{eq:unified-intro}
\alpha L(\Ric_g)+4\beta R_g-2\eps_kL(A)=2\Phi_kG_g,
\end{equation}
where
\[
\Phi_k=ce^{2F}\quad (k=1,3,4),
\qquad
\Phi_2=ce^{2F}+2\alpha\Delta_gF.
\]
Consequently, the Bochner projection of \eqref{eq:unified-intro} is $4\beta B_g=0$. This is the main reason why the method of Huang--Wan can be applied to the $k$th-mixed curvature.

Our main result is the following.

\begin{theorem}\label{thm:main-intro}
Let $(M^n,h)$, $n\geq2$, be a compact LCK manifold satisfying
\[
\cC^{(k)}_{\alpha,\beta}\equiv c,\qquad \beta\neq0.
\]
Let $\eps_k$ be given by \eqref{eq:eps-intro}. If
\[
c\neq0\qquad\text{or}\qquad \eps_k\neq0,
\]
then $h$ is K\"ahler.
\end{theorem}

As a direct consequence, Conjecture \ref{conj:kmixed} holds for compact LCK manifolds for every $k$. Since $\eps_2=\beta\neq0$, we obtain a stronger result for the second mixed curvature.

\begin{corollary}\label{cor:k2-intro}
Let $(M^n,h)$ be a  compact LCK manifold. If
\[
\cC^{(2)}_{\alpha,\beta}\equiv c,
\qquad \beta\neq0,
\]
then $h$ is K\"ahler for every $c\in\R$.
\end{corollary}

We can also describe the exceptional zero-curvature cases. For $k=1$, the exceptional parameter is $n\alpha+\beta=0$; for $k=3,4$, it is $\alpha+\beta=0$. These are genuine conformal degeneracies. In the first case the canonical K\"ahler cover is flat. For $k=3,4$ the same is true in complex dimension at least three; in complex dimension two one obtains a scalar-flat Bochner--K\"ahler alternative (see Theorem \ref{thm:exceptional}).

We also give two results for some special parameters under weaker assumptions. The Berger average singles out the line
\[
(n+1)\alpha+2\beta=0.
\]
On this line the curvature constant controls exactly the $L^2$-energy of the Chern torsion one-form. The zero constant case extends the balanced result of Chen--Tang \cite{ChenTang}.

\begin{theorem}\label{thm:critical-intro}
Let $(M^n,g)$ be a compact Hermitian manifold satisfying
\[
\cC^{(k)}_{\alpha,\beta}\equiv c,
\qquad
(n+1)\alpha+2\beta=0.
\]
Let $\eta$ be the Chern torsion one-form. Then
\[
-\frac{n(n+1)c}{\beta}\Vol(M,g)
=\int_M|\eta|^2\,dV_g,
\qquad k=1,2,
\]
and
\[
\frac{n(n+1)c}{\beta}\Vol(M,g)
=\int_M|\eta|^2\,dV_g,
\qquad k=3,4.
\]
In particular, $c/\beta\leq0$ for $k=1,2$, whereas $c/\beta\geq0$ for $k=3,4$. Moreover, $c=0$ if and only if $g$ is balanced.
\end{theorem}

For the original mixed curvature, the normalization $(\alpha,\beta)=(2,-(n+1))$ gives
\[
nc\Vol(M,g)=\int_M|\eta|^2\,dV_g.
\]
Thus $c\geq0$, and a K\"ahler solution necessarily has $c=0$. Consequently, at this parameter Conjecture \ref{conj:mixed} is equivalent to the nonexistence of a compact Hermitian metric with positive constant $\cC^{(1)}_{2,-(n+1)}$. This gives a concrete special case of the general conjecture.

For K\"ahler metrics, Chen--Tang proved that a constant mixed curvature metric has constant holomorphic sectional curvature under the two additional assumptions
\[
(n+1)\alpha+2\beta\neq0,
\qquad
(n+2)\alpha+4\beta\neq0
\]
(see \cite[Proposition 3.1]{ChenTang}). In the K\"ahler constant mixed-curvature problem considered here, we do not impose either of these two conditions in the statement below. The first one is in fact unnecessary for the complex-space-form conclusion, while the second one leads to a different geometric branch rather than an obstruction.

Indeed, the second distinguished line is
\[
(n+2)\alpha+4\beta=0.
\]
It is exactly the Bochner--K\"ahler degenerate line in the K\"ahler problem. Thus the two parameter conditions appearing in \cite[Proposition 3.1]{ChenTang} play essentially different roles.

\begin{theorem}\label{thm:kahler-intro}
Let $(M^n,g)$ be a  K\"ahler manifold with constant mixed curvature
\[
\cC_{\alpha,\beta}\equiv c,
\qquad \beta\neq0.
\]
Set $D=(n+2)\alpha+4\beta$.
\begin{enumerate}[label=\textnormal{(\arabic*)}]
\item If $D\neq0$, then $g$ has constant holomorphic sectional curvature.
\item If $D=0$, then $g$ is Bochner--K\"ahler with constant scalar curvature
\[
s=\frac{2(n+1)c}{\alpha}.
\]
Conversely, every Bochner--K\"ahler metric with this constant scalar curvature satisfies the corresponding constant mixed curvature equation.
\end{enumerate}
If $M$ is compact, the metric in the second case is locally symmetric.
\end{theorem}

Theorem \ref{thm:kahler-intro} also shows that the condition $(n+1)\alpha+2\beta\neq0$ is not needed to obtain a complex space form away from the Bochner-critical line. On the exceptional line, products of complex space forms of opposite holomorphic sectional curvatures give non-space-form examples with constant mixed curvature.

The rest of this paper is organized as follows. Section 2 recalls the Chern curvature formulas, Berger averaging, and the universal K\"ahler cover of an LCK manifold. Section 3 records two general Hermitian consequences of pointwise constant $k$th-mixed curvature. In Section 4 we derive \eqref{eq:unified-intro} and prove Bochner-flatness of the universal K\"ahler cover. Sections 5 and 6 treat the globally conformal and strict LCK branches. Section 7 proves the LCK classification. Section 8 discusses the two special parameter phenomena above.

\section{Preliminaries}

Throughout the paper, $M$ is connected and has complex dimension $n\geq2$. We use the Einstein summation convention.

\subsection{Chern curvature and the four Ricci tensors}

Let $(M^n,g)$ be a Hermitian manifold. In local holomorphic coordinates $(z^1,\ldots,z^n)$, the Chern curvature tensor is
\begin{equation}\label{eq:chern-curv}
R_{i\bar j k\bar l}
=-\partial_i\partial_{\bar j}g_{k\bar l}
+g^{p\bar q}\partial_i g_{k\bar q}\partial_{\bar j}g_{p\bar l}.
\end{equation}
The four Chern Ricci tensors are
\begin{align*}
\Ric^{(1)}_{i\bar j}&=g^{k\bar l}R_{i\bar j k\bar l},
&\Ric^{(2)}_{k\bar l}&=g^{i\bar j}R_{i\bar j k\bar l},\\
\Ric^{(3)}_{i\bar l}&=g^{k\bar j}R_{i\bar j k\bar l},
&\Ric^{(4)}_{k\bar j}&=g^{i\bar l}R_{i\bar j k\bar l}.
\end{align*}
Their scalar contractions are denoted by
\begin{equation}\label{eq:u-v}
u=g^{i\bar j}\Ric^{(1)}_{i\bar j}
=g^{k\bar l}\Ric^{(2)}_{k\bar l},
\qquad
v=g^{i\bar l}\Ric^{(3)}_{i\bar l}
=g^{k\bar j}\Ric^{(4)}_{k\bar j}.
\end{equation}
Let $T$ be the Chern torsion and let
\[
\eta=\eta_i\,dz^i,
\qquad
\eta_i=T^k_{ik},
\]
be its torsion one-form. The metric is balanced if and only if $\eta=0$. We shall use the standard identity
\begin{equation}\label{eq:uv-torsion}
\int_M(u-v)\,dV_g
=\int_M|\eta|^2\,dV_g
\end{equation}
for compact Hermitian manifolds \cite{ChenTang}.

For $0\neq X\in T^{1,0}M$, the Chern holomorphic sectional curvature is
\[
H_g(X)=\frac{R(X,\bar X,X,\bar X)}{|X|_g^4}.
\]
For fixed $\alpha,\beta\in\R$ with $\beta\neq0$, the $k$th-mixed curvature is defined by
\begin{equation}\label{eq:def-kmixed}
\cC^{(k)}_{\alpha,\beta}(X)
=\frac{\alpha}{|X|_g^2}\Ric^{(k)}(X,\bar X)+\beta H_g(X).
\end{equation}
When $k=1$, we simply write $\cC_{\alpha,\beta}$.

\subsection{Berger averaging}

Let $S^{2n-1}\subset T_p^{1,0}M$ be the unit sphere. The usual unitary average gives
\begin{equation}\label{eq:berger-H}
\frac{1}{\Vol(S^{2n-1})}\int_{S^{2n-1}}H_g(Z)\,d\theta(Z)
=\frac{u+v}{n(n+1)}.
\end{equation}
Moreover,
\begin{equation}\label{eq:berger-Ric12}
\frac{1}{\Vol(S^{2n-1})}\int_{S^{2n-1}}\Ric^{(k)}(Z,\bar Z)\,d\theta(Z)=\frac{u}{n},
\qquad k=1,2,
\end{equation}
and
\begin{equation}\label{eq:berger-Ric34}
\frac{1}{\Vol(S^{2n-1})}\int_{S^{2n-1}}\Ric^{(k)}(Z,\bar Z)\,d\theta(Z)=\frac{v}{n},
\qquad k=3,4.
\end{equation}
The same average formulas are used in Chen--Tang \cite{ChenTang}.

\subsection{Conformal change}

Let
\[
\wt g=e^{2F}g.
\]
The Chern curvature transforms as
\begin{equation}\label{eq:conf-curv}
\wt R_{i\bar j k\bar l}
=e^{2F}\bigl(R_{i\bar j k\bar l}-2F_{i\bar j}g_{k\bar l}\bigr),
\end{equation}
where $F_{i\bar j}=\partial_i\partial_{\bar j}F$. Since
$\wt g^{i\bar j}=e^{-2F}g^{i\bar j}$, the conformal factor in
\eqref{eq:conf-curv} is canceled when taking any of the four Chern
Ricci contractions. Consequently,
\begin{align}
\wt\Ric^{(1)}&=\Ric^{(1)}-2n\ddbar F,\label{eq:conf-ric1}\\
\wt\Ric^{(2)}&=\Ric^{(2)}-2(\Delta_gF)g,\label{eq:conf-ric2}\\
\wt\Ric^{(3)}&=\Ric^{(3)}-2\ddbar F,\label{eq:conf-ric3}\\
\wt\Ric^{(4)}&=\Ric^{(4)}-2\ddbar F.\label{eq:conf-ric4}
\end{align}
Also,
\begin{equation}\label{eq:conf-H}
H_{\wt g}(X)
=e^{-2F}\left(H_g(X)-2\frac{F_{i\bar j}X^i\bar X^j}{|X|_g^2}\right).
\end{equation}
Combining \eqref{eq:conf-ric1}--\eqref{eq:conf-ric4} with
\eqref{eq:conf-H}, we obtain the conformal transformation formulas
for the four mixed curvatures. For $k=1$,
\begin{equation}\label{eq:conf-mixed1}
\cC^{(1)}_{\alpha,\beta}(e^{2F}g)(X)
=e^{-2F}\left[
\cC^{(1)}_{\alpha,\beta}(g)(X)
-2(n\alpha+\beta)
\frac{F_{i\bar j}X^i\bar X^j}{|X|_g^2}
\right].
\end{equation}
For $k=2$,
\begin{equation}\label{eq:conf-mixed2}
\cC^{(2)}_{\alpha,\beta}(e^{2F}g)(X)
=e^{-2F}\left[
\cC^{(2)}_{\alpha,\beta}(g)(X)
-2\alpha\Delta_gF
-2\beta\frac{F_{i\bar j}X^i\bar X^j}{|X|_g^2}
\right].
\end{equation}
For $k=3,4$,
\begin{equation}\label{eq:conf-mixed34}
\cC^{(k)}_{\alpha,\beta}(e^{2F}g)(X)
=e^{-2F}\left[
\cC^{(k)}_{\alpha,\beta}(g)(X)
-2(\alpha+\beta)
\frac{F_{i\bar j}X^i\bar X^j}{|X|_g^2}
\right].
\end{equation}
These formulas explain the coefficients in \eqref{eq:eps-intro}. For
$k=1$, the direction-dependent complex Hessian term disappears when
$n\alpha+\beta=0$, and hence
\begin{equation}\label{eq:conf-cov1}
\cC^{(1)}_{\alpha,-n\alpha}(e^{2F}g)
=e^{-2F}\cC^{(1)}_{\alpha,-n\alpha}(g).
\end{equation}
Similarly, for $k=3,4$, the corresponding exceptional parameter is
$\alpha+\beta=0$, and
\begin{equation}\label{eq:conf-cov34}
\cC^{(k)}_{\alpha,-\alpha}(e^{2F}g)
=e^{-2F}\cC^{(k)}_{\alpha,-\alpha}(g).
\end{equation}
In contrast, there is no conformally degenerate parameter for the
second mixed curvature. Indeed, by \eqref{eq:conf-mixed2}, the
coefficient of the direction-dependent complex Hessian term is always
$\beta$, which is assumed to be nonzero. The term involving
$\Delta_gF$ is independent of the direction $X$ and can be absorbed
into the pointwise curvature function.

\subsection{LCK manifolds and the universal K\"ahler cover}

Let $(M,h)$ be LCK. On the universal covering $\pi:\wt M\to M$, there exist a K\"ahler metric $g$ and a real-valued function $F$ such that
\begin{equation}\label{eq:univ-cover}
\pi^*h=e^{2F}g.
\end{equation}
Every deck transformation $\gamma$ acts by a holomorphic homothety of $g$. More precisely, for each fixed $\gamma$ there is a positive constant $r_\gamma$, depending only on $\gamma$, such that
\begin{equation}\label{eq:deck}
\gamma^*g=r_\gamma^2g,
\qquad
F\circ\gamma=F-\log r_\gamma.
\end{equation}
The Lee form is exact if and only if $r_\gamma=1$ for every deck transformation \cite{DO,Vaisman}.

\subsection{The $L$-operator and K\"ahler curvature decomposition}

Let $(V,q)$ be a Hermitian vector space. For a $(1,1)$-tensor $S$, define
\begin{align}
(G_q)_{i\bar j k\bar l}
&=q_{i\bar j}q_{k\bar l}+q_{i\bar l}q_{k\bar j},\label{eq:Gdef}\\
(L_q(S))_{i\bar j k\bar l}
&=S_{i\bar j}q_{k\bar l}+S_{k\bar j}q_{i\bar l}
 +S_{i\bar l}q_{k\bar j}+S_{k\bar l}q_{i\bar j}.
\label{eq:Ldef}
\end{align}
We omit the subscript when the metric is clear.

\begin{lemma}\label{lem:L-inj}
If $L_q(S)=\mu G_q$, then
\[
S=\frac{\mu}{2}q.
\]
In particular, $L_q$ is injective.
\end{lemma}

\begin{proof}
Contracting the last two indices gives
\[
(n+2)S+(\tr_qS)q=\mu(n+1)q.
\]
Taking one more trace yields $\tr_qS=n\mu/2$, and the result follows.
\end{proof}

For an algebraic K\"ahler curvature tensor,
\begin{equation}\label{eq:BK-decomp}
R=B(R)+\frac{1}{n+2}L(\Ric^0)+\frac{s}{n(n+1)}G,
\end{equation}
where $B(R)$ is the Bochner tensor and $\Ric^0=\Ric-\frac{s}{n}g$. This is the standard $U(n)$-decomposition of a K\"ahler curvature tensor into its Bochner, trace-free Ricci, and scalar parts. The Bochner part is totally trace-free; in particular,
\[
g^{k\bar l}B(R)_{i\bar j k\bar l}=0.
\]
Indeed, for any Hermitian $(1,1)$-tensor $S$,
\begin{equation}\label{eq:L-contraction}
g^{k\bar l}L(S)_{i\bar j k\bar l}
=(n+2)S_{i\bar j}+(\tr_g S)g_{i\bar j},
\end{equation}
and $g^{k\bar l}G_{i\bar j k\bar l}=(n+1)g_{i\bar j}$. Thus the second and third terms in \eqref{eq:BK-decomp} recover respectively the trace-free Ricci tensor and the scalar part of $R$, while $B(R)$ is the remaining curvature component which is invisible to Ricci contraction. In particular, if $B(R)=0$, then the full K\"ahler curvature tensor is determined by its Ricci tensor. Notice that Bochner-flatness does not mean that $R$ itself vanishes.

We use two standard global facts. The compact smooth classification of Bochner--K\"ahler manifolds implies that every connected compact smooth Bochner--K\"ahler manifold is locally symmetric; equivalently, its universal covering is one of the symmetric Bochner--K\"ahler models described by Kamishima and Bryant \cite{Kam94,Kam05,Bryant}. Here the local symmetry is a global rigidity consequence of the compact smooth classification, and is not a formal consequence of $B(R)=0$ for an arbitrary local Bochner--K\"ahler metric. A connected compact strict LCK manifold which is Tricerri--Vanhecke Bochner-flat has, after a constant rescaling, a canonical K\"ahler cover of the following form (see \cite[Proposition 5.1]{HuangWan}; see also \cite{Kam06,Fried}):
\[
(\C^n\setminus\{0\},g_{\mathrm{Euc}}),
\]
and some deck transformation has the form $z\mapsto rUz$, where $r>0$ is a constant with $r\neq1$ and $U\in U(n)$.

\section{Two general Hermitian consequences}

We first give two observations which do not require the LCK condition.

\subsection{The K\"ahler-symmetrized Chern curvature}

Define the K\"ahler symmetrization of the Chern curvature by
\begin{equation}\label{eq:Ksym}
K_{i\bar j k\bar l}
=\frac14\bigl(
R_{i\bar j k\bar l}+R_{k\bar j i\bar l}
+R_{i\bar l k\bar j}+R_{k\bar l i\bar j}
\bigr).
\end{equation}
It has the algebraic K\"ahler symmetries and satisfies
\[
K(X,\bar X,X,\bar X)=R(X,\bar X,X,\bar X).
\]

\begin{proposition}\label{prop:general-polarization}
Suppose that at a point $p\in M$ the $k$th-mixed curvature is independent of the direction:
\[
\cC^{(k)}_{\alpha,\beta}(X)=\varphi(p)
\qquad\text{for all }0\neq X\in T_p^{1,0}M.
\]
Then
\begin{equation}\label{eq:general-polarization}
\alpha L(\Ric^{(k)})+4\beta K=2\varphi G.
\end{equation}
Consequently, the Bochner component of $K$ vanishes.
\end{proposition}

\begin{proof}
Multiplying the curvature identity by $|X|^4$ gives
\[
\alpha\Ric^{(k)}(X,\bar X)|X|^2
+\beta R(X,\bar X,X,\bar X)
=\varphi|X|^4.
\]
Polarization gives \eqref{eq:general-polarization} (see \cite[Lemma 3.1]{ChenTang}; compare also \cite{CCN,HuangWan}). The tensors $L(\Ric^{(k)})$ and $G$ have no component in the Bochner summand of the complexified K\"ahler curvature module, so $\beta B(K)=0$. Since $\beta\neq0$, $B(K)=0$.
\end{proof}

\begin{remark}
Proposition \ref{prop:general-polarization} only controls the K\"ahler-symmetric part of the Chern curvature. The remaining components contain torsion information, and this is an obstruction to applying the Bochner--K\"ahler rigidity directly to a general Hermitian manifold.
\end{remark}

\subsection{The averaging-critical parameter}

\begin{proof}[Proof of Theorem \ref{thm:critical-intro}]
Averaging \eqref{eq:def-kmixed} and using \eqref{eq:berger-H}--\eqref{eq:berger-Ric34}, we obtain
\begin{equation}\label{eq:avg12}
((n+1)\alpha+\beta)u+\beta v=n(n+1)c,
\qquad k=1,2,
\end{equation}
and
\begin{equation}\label{eq:avg34}
((n+1)\alpha+\beta)v+\beta u=n(n+1)c,
\qquad k=3,4.
\end{equation}
Since $(n+1)\alpha+2\beta=0$, we have $(n+1)\alpha+\beta=-\beta$. Hence
\begin{align}
u-v&=-\frac{n(n+1)c}{\beta},&& k=1,2,\label{eq:uv-critical12}\\
u-v&=\frac{n(n+1)c}{\beta},&& k=3,4.\label{eq:uv-critical34}
\end{align}
Integrating these identities and using \eqref{eq:uv-torsion}, we obtain
\begin{align}
-\frac{n(n+1)c}{\beta}\Vol(M,g)
&=\int_M|\eta|^2\,dV_g,&& k=1,2,\label{eq:energy12}\\
\frac{n(n+1)c}{\beta}\Vol(M,g)
&=\int_M|\eta|^2\,dV_g,&& k=3,4.\label{eq:energy34}
\end{align}
The sign assertions follow immediately.

If $c=0$, the right-hand side vanishes, so $\eta=0$ and $g$ is balanced. Conversely, if $g$ is balanced, then the integral of $u-v$ vanishes. Since \eqref{eq:uv-critical12} or \eqref{eq:uv-critical34} shows that $u-v$ is constant, it follows that $c=0$.
\end{proof}

\begin{corollary}\label{cor:normalized-critical}
For $(\alpha,\beta)=(2,-(n+1))$, a compact Hermitian metric with constant first or second mixed curvature satisfies
\[
nc\Vol(M,g)=\int_M|\eta|^2\,dV_g.
\]
Hence $c\geq0$, with equality if and only if $g$ is balanced.
\end{corollary}

\begin{remark}
For $k=1$, every K\"ahler solution at $(\alpha,\beta)=(2,-(n+1))$ has $c=0$. Thus Conjecture \ref{conj:kmixed} at this parameter is equivalent to the nonexistence of a compact Hermitian metric with $\cC^{(1)}_{2,-(n+1)}\equiv c>0$.
\end{remark}

\section{The universal K\"ahler cover}

Let $(M^n,h)$ be a compact LCK manifold with
\[
\cC^{(k)}_{\alpha,\beta}(h)\equiv c.
\]
Write the universal K\"ahler cover as in \eqref{eq:univ-cover}:
\[
\pi^*h=e^{2F}g.
\]
Since $g$ is K\"ahler, all four Chern Ricci tensors of $g$ coincide with $\Ric_g$. Set
\[
A_{i\bar j}=F_{i\bar j}.
\]

\begin{proposition}\label{prop:cover-identity}
On the universal K\"ahler cover,
\begin{equation}\label{eq:cover-unified}
\alpha L(\Ric_g)+4\beta R_g-2\eps_kL(A)=2\Phi_kG_g,
\end{equation}
where
\begin{equation}\label{eq:Phi}
\Phi_k=
\begin{cases}
ce^{2F},&k=1,3,4,\\
ce^{2F}+2\alpha\Delta_gF,&k=2,
\end{cases}
\end{equation}
and
\[\eps_1=n\alpha+\beta,\qquad \eps_2=\beta,\qquad \eps_3=\eps_4=\alpha+\beta.\]
\end{proposition}

\begin{proof}
For $k=1$, formulas \eqref{eq:conf-ric1} and \eqref{eq:conf-H} give, after multiplication by $e^{2F}|X|_g^4$,
\[
\alpha\Ric_g(X,\bar X)|X|^2
+\beta R_g(X,\bar X,X,\bar X)
-2(n\alpha+\beta)A(X,\bar X)|X|^2
=ce^{2F}|X|^4.
\]
Polarization yields \eqref{eq:cover-unified} with $\eps_1=n\alpha+\beta$.

For $k=2$, using \eqref{eq:conf-ric2},
\[
\alpha\Ric_g(X,\bar X)|X|^2
+\beta R_g(X,\bar X,X,\bar X)
-2\beta A(X,\bar X)|X|^2
=(ce^{2F}+2\alpha\Delta_gF)|X|^4,
\]
which gives the asserted formula. The cases $k=3,4$ follow in the same way from \eqref{eq:conf-ric3}--\eqref{eq:conf-ric4}.
\end{proof}

\begin{corollary}\label{cor:cover-BK}
The universal K\"ahler metric $g$ is Bochner--K\"ahler. Equivalently, $h$ is Tricerri--Vanhecke Bochner-flat.
\end{corollary}

\begin{proof}
Recall that every Hermitian $(1,1)$-tensor $S$ can be written as
\[
S=S^0+\frac{\tr_g S}{n}g.
\]
Since $L(g)=2G$, we have
\[
L(S)=L(S^0)+\frac{2\tr_g S}{n}G.
\]
Hence $L(S)$ has only trace-free Ricci and scalar components in the K\"ahler curvature decomposition \eqref{eq:BK-decomp}, and in particular it has no Bochner component. Therefore both $L(\Ric_g)$ and $L(A)$ have zero Bochner component, while $G_g$ is purely scalar.

Taking the Bochner projection of \eqref{eq:cover-unified}, the only term which can contribute is $4\beta R_g$. Thus
\[
4\beta B(R_g)=0.
\]
Since $\beta\neq0$, we obtain $B(R_g)=0$, and hence the K\"ahler metric $g$ is Bochner--K\"ahler.

Finally, for a K\"ahler metric the Tricerri--Vanhecke Bochner tensor agrees with the ordinary K\"ahler Bochner tensor, while its vanishing is preserved under Hermitian conformal changes \cite{TV,Kam06,HuangWan}. Since $\pi^*h=e^{2F}g$, the equality $B(R_g)=0$ is therefore equivalent to the vanishing of the Tricerri--Vanhecke Bochner tensor of $h$.
\end{proof}

\begin{remark}
This is the key step where the curvature condition is used to apply the Bochner--K\"ahler theory. The extra $k$th-Ricci contribution and every conformal Hessian term occur through $L(\cdot)$, while the Laplacian term for $k=2$ is a multiple of $G_g$. None of these terms has a Bochner component. This is the main reason why the Huang--Wan argument \cite{HuangWan} works for all four $k$th-mixed curvatures.
\end{remark}

\section{The globally conformal K\"ahler branch}

Assume that the Lee form of $h$ is exact. Then
\[
h=e^{2F}g_0
\]
for a K\"ahler metric $g_0$ on the compact manifold $M$. By Corollary \ref{cor:cover-BK}, $g_0$ is Bochner--K\"ahler. The compact smooth classification of Bochner--K\"ahler manifolds \cite{Kam94,Kam05,Bryant} implies that $g_0$ is locally symmetric. Hence
\begin{equation}\label{eq:parallel-R}
\nabla R_{g_0}=0,
\qquad
\nabla\Ric_{g_0}=0.
\end{equation}

\begin{proposition}\label{prop:GCK}
Let $(M^n,h)$ be compact and globally conformal K\"ahler with constant $k$th-mixed curvature. If
\[
\eps_k\neq0
\qquad\text{or}\qquad
c\neq0,
\]
then $h$ is K\"ahler.
\end{proposition}

\begin{proof}
The identity \eqref{eq:cover-unified} descends to $M$ with $g=g_0$ and $A=\ddbar F$.

Assume first that $\eps_k\neq0$. Differentiating \eqref{eq:cover-unified} in a $(1,0)$-direction and using \eqref{eq:parallel-R} gives
\[
L(\nabla_mA)=-\frac{(\Phi_k)_m}{\eps_k}G_{g_0}.
\]
By Lemma \ref{lem:L-inj},
\begin{equation}\label{eq:nablaA}
\nabla_mA_{i\bar j}
=-\frac{(\Phi_k)_m}{2\eps_k}(g_0)_{i\bar j}.
\end{equation}
Because $A_{i\bar j}=\nabla_i\nabla_{\bar j}F$ and $g_0$ is K\"ahler, holomorphic covariant derivatives commute:
\[
\nabla_mA_{i\bar j}=\nabla_iA_{m\bar j}.
\]
Combining this with \eqref{eq:nablaA},
\[
(\Phi_k)_m(g_0)_{i\bar j}
=(\Phi_k)_i(g_0)_{m\bar j}.
\]
At a fixed point choose a unitary frame. For each $m$, choose $i\neq m$ and put $j=i$. Then $(\Phi_k)_m=0$, so $\partial\Phi_k=0$. Since $\Phi_k$ is real-valued, also $\bar\partial\Phi_k=0$, and hence $d\Phi_k=0$. Equation \eqref{eq:nablaA} now gives $\nabla^{1,0}A=0$. Since $A=\ddbar F$ is Hermitian and the K\"ahler connection is compatible with complex conjugation, its conjugate equation gives $\nabla^{0,1}A=0$. Thus $\nabla A=0$.

The real $(1,1)$-form $\sqrt{-1}\ddbar F$ is therefore parallel and exact. A parallel form is coclosed, and integration by parts gives
\[
\int_M|\ddbar F|^2\,dV_{g_0}=0.
\]
Thus $A=0$. Taking the trace, $\Delta_{g_0}F=0$, so $F$ is constant.

It remains to consider $\eps_k=0$. Since $\eps_2=\beta\neq0$, this can occur only for $k=1,3,4$. In this case \eqref{eq:cover-unified} becomes
\[
\alpha L(\Ric_{g_0})+4\beta R_{g_0}=2ce^{2F}G_{g_0}.
\]
The left-hand side is parallel. If $c\neq0$, then $d(e^{2F})=0$, and hence $F$ is constant.
\end{proof}

\section{The strict LCK branch}

We now assume that the Lee form is not exact. By Corollary \ref{cor:cover-BK}, $h$ is Tricerri--Vanhecke Bochner-flat. We use the following LCK Bochner-flat uniformization result from Huang--Wan, which combines Kamishima's uniformization theorem with Fried's classification of incomplete closed similarity manifolds (see \cite[Proposition 5.1]{HuangWan}; see also \cite{Kam06,Fried}).

\begin{proposition}[Huang--Wan, Kamishima--Fried]\label{prop:HW-uniformization}
Let $(M^n,h)$ be a connected compact strict LCK manifold, $n\geq2$, whose Tricerri--Vanhecke Bochner tensor vanishes. After multiplying the canonical K\"ahler metric on the universal cover by a positive constant, one may identify
\[
(\wt M,g)=(\C^n\setminus\{0\},g_{\mathrm{Euc}}).
\]
Every deck transformation is of the form
\[
\gamma(z)=r_\gamma U_\gamma z,
\qquad r_\gamma>0,
\quad U_\gamma\in U(n),
\]
and $r_\gamma\neq1$ for at least one deck transformation.
\end{proposition}

Under this identification,
\[
\pi^*h=e^{2F}g_{\mathrm{Euc}}.
\]
For a deck transformation $\gamma(z)=rUz$, one has $\gamma^*g_{\mathrm{Euc}}=r^2g_{\mathrm{Euc}}$. Since $\gamma^*(\pi^*h)=\pi^*h$, formula \eqref{eq:deck} gives
\[
F\circ\gamma=F-\log r,
\qquad\text{that is,}\qquad
F(rUz)=F(z)-\log r.
\]
In the strict LCK case one can choose such a transformation with $r\neq1$.

\begin{proposition}\label{prop:strict}
Let $(M^n,h)$ be a connected compact strict LCK manifold with constant $k$th-mixed curvature. Then necessarily
\[
c=0,
\qquad
\eps_k=0.
\]
In particular, no compact strict LCK metric has constant second mixed curvature.
\end{proposition}

\begin{proof}
On the flat universal K\"ahler cover, \eqref{eq:cover-unified} reduces to
\begin{equation}\label{eq:flat-identity}
-2\eps_kL(A)=2\Phi_kG_{\mathrm{Euc}}.
\end{equation}
Suppose first that $\eps_k\neq0$. Lemma \ref{lem:L-inj} gives
\begin{equation}\label{eq:A-lambda}
A=\lambda g_{\mathrm{Euc}}
\end{equation}
for a real-valued function $\lambda$. In Euclidean coordinates,
\[
F_{i\bar j}=\lambda\delta_{ij}.
\]
Differentiating with respect to $z^m$ gives
\[
F_{mi\bar j}=\lambda_m\delta_{ij},
\]
while differentiating $F_{m\bar j}=\lambda\delta_{mj}$ with respect to $z^i$ gives
\[
F_{im\bar j}=\lambda_i\delta_{mj}.
\]
Since ordinary holomorphic derivatives commute,
\[
\lambda_m\delta_{ij}=\lambda_i\delta_{mj}.
\]
For fixed $m$, choose $i\neq m$ and set $j=i$. Then $\lambda_m=0$. Hence $\partial\lambda=0$, and since $\lambda$ is real-valued, $d\lambda=0$.

Choose a deck transformation $\gamma(z)=rUz$ with $r\neq1$. Since $F\circ\gamma=F-\log r$ and $r$ is constant,
\[
\gamma^*A
=\ddbar(F\circ\gamma)
=\ddbar F=A.
\]
On the other hand, pullback commutes with multiplication by a function, so \eqref{eq:A-lambda} and the constancy of $\lambda$ give
\[
\gamma^*A
=\gamma^*(\lambda g_{\mathrm{Euc}})
=(\lambda\circ\gamma)\gamma^*g_{\mathrm{Euc}}
=\lambda r^2g_{\mathrm{Euc}}.
\]
Thus $\lambda(r^2-1)=0$, and hence $\lambda=0$. Therefore
\[
A=0.
\]
In particular, $\Delta F=0$. Equation \eqref{eq:flat-identity} now gives $ce^{2F}=0$, so $c=0$.

Since $A=\ddbar F=0$, the function $F$ is real pluriharmonic on $\C^n\setminus\{0\}$. For $n\geq2$ the punctured space is simply connected, so $F=\operatorname{Re}G$ for a holomorphic function $G$ on $\C^n\setminus\{0\}$, after adding a constant to $G$ if necessary. Hartogs' extension theorem extends $G$ holomorphically across the origin. The automorphy relation implies
\[
\operatorname{Re}\bigl(G(rUz)-G(z)\bigr)=-\log r.
\]
A holomorphic function with constant real part is constant, so $G(rUz)-G(z)$ is constant. Since the identity extends to $z=0$ and $rU0=0$, this constant is zero. Hence $\log r=0$, contradicting $r\neq1$.

Therefore $\eps_k\neq0$ is impossible in the strict LCK case. Hence $\eps_k=0$. This can occur only for $k=1,3,4$. Then \eqref{eq:flat-identity} becomes
\[
0=2ce^{2F}G_{\mathrm{Euc}},
\]
so $c=0$.
\end{proof}

\section{Proof of the main LCK rigidity theorem}

We now combine the two cases above.

\begin{proof}[Proof of Theorem \ref{thm:main-intro}]
If the Lee form is exact, the result follows from Proposition \ref{prop:GCK}. If the Lee form is not exact, Proposition \ref{prop:strict} shows that necessarily $c=0$ and $\eps_k=0$. This proves Theorem \ref{thm:main-intro}.
\end{proof}

Conjecture \ref{conj:kmixed} for compact LCK manifolds follows immediately from Theorem \ref{thm:main-intro}.

\begin{proof}[Proof of Corollary \ref{cor:k2-intro}]
For $k=2$ one has $\eps_2=\beta\neq0$. Hence Theorem \ref{thm:main-intro} applies for every value of the curvature constant $c$.
\end{proof}

The exceptional zero-curvature cases can be described on the canonical K\"ahler cover.

\begin{theorem}\label{thm:exceptional}
Let $(M^n,h)$ be a connected compact LCK manifold with
\[
\cC^{(k)}_{\alpha,\beta}\equiv0,
\qquad
\eps_k=0.
\]
Then:
\begin{enumerate}[label=\textnormal{(\arabic*)}]
\item If $k=1$, the canonical K\"ahler metric on the universal cover is flat.
\item If $k=3$ or $4$ and $n\geq3$, the canonical K\"ahler metric on the universal cover is flat.
\item If $k=3$ or $4$ and $n=2$, the canonical K\"ahler metric is scalar-flat Bochner--K\"ahler. If the LCK metric is globally conformal K\"ahler, this K\"ahler metric is locally symmetric; if the LCK metric is strict, the canonical K\"ahler cover is flat.
\end{enumerate}
\end{theorem}

\begin{proof}
On the universal K\"ahler cover, \eqref{eq:cover-unified} becomes
\begin{equation}\label{eq:exceptional-cover}
\alpha L(\Ric_g)+4\beta R_g=0.
\end{equation}
For $k=1$, $\eps_1=0$ means $\beta=-n\alpha$. Since $\beta\neq0$, $\alpha\neq0$, and \eqref{eq:exceptional-cover} becomes
\[
L(\Ric_g)-4nR_g=0.
\]
Taking a Ricci contraction gives
\[
(n+2)\Ric_g+s_g g-4n\Ric_g=0,
\]
or
\[
s_g g=(3n-2)\Ric_g.
\]
Tracing once more yields $s_g=0$, hence $\Ric_g=0$ and then $R_g=0$.

For $k=3,4$, $\eps_k=0$ means $\beta=-\alpha$, and \eqref{eq:exceptional-cover} becomes
\[
L(\Ric_g)-4R_g=0.
\]
A Ricci contraction gives
\[
(n-2)\Ric_g+s_g g=0.
\]
Taking the trace yields $s_g=0$. If $n\geq3$, then $\Ric_g=0$ and consequently $R_g=0$.

If $n=2$, the same equation says $s_g=0$ and
\[
R_g=\frac14L(\Ric_g),
\]
which is precisely the scalar-flat Bochner--K\"ahler condition. In the globally conformal K\"ahler case this metric descends to a compact Bochner--K\"ahler metric, hence is locally symmetric by the compact classification \cite{Kam94,Kam05,Bryant}. In the strict LCK case flatness follows from Proposition \ref{prop:HW-uniformization}.
\end{proof}

\begin{remark}\label{rem:sharp}
These exceptional parameters are sharp. The standard isosceles Hopf metric satisfies
\[
\cC^{(1)}_{\alpha,-n\alpha}\equiv0,
\]
and similarly the conformal covariance \eqref{eq:conf-cov34} produces zero-curvature non-K\"ahler examples for $k=3,4$ at $\alpha+\beta=0$. Thus the hypothesis in Theorem \ref{thm:main-intro} cannot in general be weakened when $c=0$.
\end{remark}

\section{Special parameter phenomena}

Finally, we give two calculations for special parameters which are independent of the LCK classification.

\subsection{The Bochner-critical line in the K\"ahler case}

Let $(M^n,g)$ be K\"ahler and suppose
\[
\cC_{\alpha,\beta}\equiv c.
\]
All four Chern Ricci tensors coincide, so the choice of $k$ is irrelevant. Polarization gives
\begin{equation}\label{eq:K-polar}
\alpha L(\Ric)+4\beta R=2cG.
\end{equation}
Taking a Ricci contraction and using
\[
\tr_{3,4}L(\Ric)=(n+2)\Ric+s g,
\qquad
\tr_{3,4}G=(n+1)g,
\]
we obtain
\begin{equation}\label{eq:K-Ric-eq}
\bigl((n+2)\alpha+4\beta\bigr)\Ric+\alpha s g
=2(n+1)c g.
\end{equation}

Set
\[
D=(n+2)\alpha+4\beta.
\]

\begin{proof}[Proof of Theorem \ref{thm:kahler-intro}]
Assume first that $D\neq0$. Taking the trace-free part of \eqref{eq:K-Ric-eq} gives
\[
D\Ric^0=0,
\]
so $\Ric=\lambda g$ for a real-valued function $\lambda$. Let $\rho$ be the K\"ahler Ricci form. Since $\rho=\lambda\omega$ and $d\rho=0=d\omega$,
\[
d\lambda\wedge\omega=0.
\]
Wedging with $\omega^{n-2}$ gives
\[
d\lambda\wedge\omega^{n-1}=0.
\]
At each point the Lefschetz map
\[
L^{n-1}:\Lambda^1T^*M\longrightarrow\Lambda^{2n-1}T^*M,
\qquad
\xi\longmapsto\xi\wedge\omega^{n-1},
\]
is an isomorphism. Indeed, write $\xi=\iota_X\omega$. Then
\[
\iota_X(\omega^n)=n(\iota_X\omega)\wedge\omega^{n-1}
=n\xi\wedge\omega^{n-1}.
\]
If $\xi\wedge\omega^{n-1}=0$, then $\iota_X(\omega^n)=0$. Since $\omega^n$ is a volume form, $X=0$ and hence $\xi=0$; equality of dimensions gives the isomorphism. Thus this is only a pointwise symplectic linear-algebra fact, not the global Hard Lefschetz theorem. Hence $d\lambda=0$, so the Einstein factor is constant. Returning to \eqref{eq:K-polar} and using $L(g)=2G$, we obtain
\[
R=\frac{c-\alpha\lambda}{2\beta}G,
\]
and therefore the holomorphic sectional curvature is constant.

Now suppose $D=0$. Then $4\beta=-(n+2)\alpha$, and $\alpha\neq0$. Equation \eqref{eq:K-polar} becomes
\begin{equation}\label{eq:BKcritical}
L(\Ric)-(n+2)R=\frac{2c}{\alpha}G.
\end{equation}
Writing $\Ric=\Ric^0+\frac{s}{n}g$ and using \eqref{eq:BK-decomp} together with $L(g)=2G$, we obtain
\begin{align*}
L(\Ric)-(n+2)R
&=L(\Ric^0)+\frac{2s}{n}G
 -(n+2)B(R)-L(\Ric^0)
 -\frac{(n+2)s}{n(n+1)}G\\
&=-(n+2)B(R)+\frac{s}{n+1}G.
\end{align*}
Comparing the Bochner and scalar components with \eqref{eq:BKcritical} gives
\[
B(R)=0,
\qquad
s=\frac{2(n+1)c}{\alpha}.
\]
Conversely, if $B(R)=0$ and $s=2(n+1)c/\alpha$, the same computation gives
\[
L(\Ric)-(n+2)R=\frac{2c}{\alpha}G.
\]
Multiplying by $\alpha$ and using $4\beta=-(n+2)\alpha$ recovers \eqref{eq:K-polar}. If $M$ is compact, local symmetry follows from the compact Bochner--K\"ahler classification \cite{Kam94,Kam05,Bryant}.
\end{proof}

\begin{remark}
The proof of Theorem \ref{thm:kahler-intro}(1) requires only
\[
(n+2)\alpha+4\beta\neq0.
\]
Thus the additional condition $(n+1)\alpha+2\beta\neq0$ in the earlier K\"ahler statement is not needed (see \cite[Proposition 3.1]{ChenTang}). The point is that once $\Ric=\lambda g$ is known, the closedness of the K\"ahler Ricci form and the pointwise Lefschetz isomorphism force $\lambda$ to be constant, even on the averaging-critical line. If $(n+1)\alpha+2\beta=0$, a complex space form may have zero mixed curvature although its holomorphic sectional curvature is nonzero.
\end{remark}

\begin{example}\label{ex:product}
The exceptional line in Theorem \ref{thm:kahler-intro} cannot be removed. Let
\[
M=M_+^p\times M_-^q,
\qquad p+q=n,
\]
where the two factors are K\"ahler space forms with holomorphic sectional curvatures $\kappa$ and $-\kappa$, respectively. For a unit vector $X=X_++X_-$, put $|X_+|^2=t$. Then
\[
H(X)=\kappa(2t-1)
\]
and
\[
\Ric(X,\bar X)
=\frac{p+1}{2}\kappa t
-\frac{q+1}{2}\kappa(1-t).
\]
Hence $\cC_{\alpha,\beta}(X)$ is independent of $t$ exactly when
\[
(n+2)\alpha+4\beta=0.
\]
The resulting constant is
\[
c=\frac{\alpha\kappa}{4}(p-q).
\]
For instance, on $\CP^2\times\Sigma_g$, $g(\Sigma_g)>1$, with the factors normalized to have holomorphic sectional curvatures $+1$ and $-1$, respectively,
\[
\cC_{4,-5}\equiv1,
\]
although the holomorphic sectional curvature is not constant.
\end{example}

\subsection{Three distinguished lines for the first mixed curvature}

For $k=1$, the preceding calculations single out three different parameter lines:
\[
n\alpha+\beta=0,
\qquad
(n+1)\alpha+2\beta=0,
\qquad
(n+2)\alpha+4\beta=0.
\]
These three lines have different meanings. The first is the conformal-critical line, as shown by \eqref{eq:conf-cov1}. The second is the averaging-critical line, where Theorem \ref{thm:critical-intro} turns the curvature constant into the torsion energy. The third is the Bochner-critical line in K\"ahler geometry, where the constant mixed curvature equation is exactly the constant-scalar-curvature Bochner--K\"ahler equation.

This distinction may also be useful for the general Hermitian conjecture. In particular, at
\[
(\alpha,\beta)=(2,-(n+1)),
\]
Proposition \ref{prop:general-polarization} and Corollary \ref{cor:normalized-critical} give simultaneously
\[
B(K)=0,
\qquad
\int_M|\eta|^2\,dV_g=nc\Vol(M,g).
\]
Thus a possible counterexample at this parameter must have $c>0$ and a prescribed positive torsion energy. It is natural to ask whether the remaining nonsymmetric Chern curvature identities can rule out this case.

\section*{Acknowledgements}
The authors would like to thank Professor Fangyang Zheng for his long-term support and help. The main ideas of this work were developed by the authors. During the preparation of this manuscript, ChatGPT (OpenAI) was used for language polishing and structural organization. The authors carefully reviewed and revised the manuscript and take full responsibility for all mathematical statements, arguments, and conclusions.

\end{document}